\documentclass[12pt]{amsart}

\usepackage{melanie, pdfsync}

\usepackage{fullpage}

\newtheorem{nts}{Note to self}

\newcommand\highlight[1]{#1}

\newcommand\la{\lambda}
\renewcommand\l{\lambda}

\newcommand\wbar{\overline{w}}
\renewcommand\k{{\bf k}}

\newcommand\cM{\mathcal{M}}
\newcommand{\var}{\rm{Var}}
\newcommand{\zzeta}{\mathbf{Z}}
\newcommand\Kbar{\overline{K}}

\title{Counting polynomials over finite fields with given root multiplicities}

\author{Ayah Almousa and Melanie Matchett Wood}
\address{Department of Mathematics\\
University of Wisconsin-Madison \\ 480 Lincoln Drive \\
Madison, WI 53705 USA\\
and
American Institute of Mathematics\\360 Portage Ave \\
Palo Alto, CA 94306-2244 USA} 
\email{mmwood@math.wisc.edu}

\begin{document}
\begin{abstract}
 We give formulas for the number of polynomials over a finite field with given root multiplicities, in particular in cases when the formula is surprisingly simple (a power of $q$).  Besides this concrete interpretation, we also prove an analogous result on configuration spaces in the Grothendieck ring of varieties, suggesting new homological stabilization conjectures for configuration spaces of the plane.
\end{abstract}

\maketitle

\section{Introduction}
Given a finite field $\F_q$, a monic polynomial $f\in\F_q[x]$ factors into linear 
factors
$(x-\alpha_1)^{e_1}\cdots (x-\alpha_t)^{e_t}$ over the algebraic closure $\bar{\F_q}$ (with $\alpha_i\in \bar{\F_q}$ distinct).
To $f$, we can associate the partition $P(f)=e_0\cdots e_t$ (using multiplicative notation for partitions).  For a partition $\lambda$, we define
$$
w_\lambda:=\#\{\textrm{monic } f\in\F_q[x] | P(f)=\lambda\}.
$$
For example,
$w_{12}=q^2-q,$
where the subscript ``$12$'' denotes the partition with two elements $1,2$.  The number of square-free monic polynomials of degree $n\geq 2$ over $\F_q$ is $w_{1^n}=q^n-q^{n-1},$
a well-known fact.  

For two partitions $\lambda, \lambda'$, we define the refinement ordering $\lambda\leq \lambda'$ if $\lambda$ can be partitioned into subsets that add to the elements of $\lambda'$, so for example 
$11247\leq 357$
 (see Section~\ref{S:notation}).
We define
$$
\wbar_\lambda =\sum_{\lambda'\geq \lambda} w_{\lambda'},
$$
so for example $\wbar_{1^n}=q^n$ as $\{\lambda'\geq 1^n\}$ is the set of all partitions of $n$.
Also, $\wbar_{1^n2}=\wbar_{1^{n+2}}-w_{1^{n+2}}=q^{n+1}$.
We have that $\wbar_{1^nd}$ counts polynomials with at least one root with multiplicity at least $d$, and we will see that
$
\wbar_{1^nd}=q^{n+1}.
$
Similarly, $\wbar_{1^n22}$ counts polynomials with either at least two roots with multiplicity at least $2$ or 
at least one root with multiplicity at least $4$, and $\wbar_{1^n22}=q^{n+2}$.
As another example $\wbar_{12259}=q^{5}$.  These examples, and many more, lead to the natural conjecture that
\begin{equation}\label{E:conj}
\wbar_\lambda \stackrel{?}{=} q^{|\lambda|},
\end{equation}
for all $\lambda$, given for example as a comment to \cite{ellenberg-blog-Lconj}.  In fact,
this conjecture is false as $w_{11223}=q^5+q^2-q$, as pointed out in \cite[Section 2]{vakil-wood}.

In this paper, we address the question of when Equation~\eqref{E:conj} is true and we prove the following.
\begin{theorem}\label{T:main}
 For integers $m\geq -1,$ and $k\geq 0$, and $b_i,e_i\geq 1$ for $0\leq i\leq m$, such that each $1\leq i\leq m$, we have
$b_i\geq\sum_{j<i} e_jb_j$, we have
$$
 \wbar_{1^k{b_0}^{e_0}\cdots{b_m}^{e_m}} = q^{k+\sum_i e_i}.
$$ 
\end{theorem}
This theorem proves a large number of cases of when Equation~\eqref{E:conj} holds, including those mentioned above.
In the special cases when $m=1$ or both $m=2$ and $e_0=1$, Theorem~\ref{T:main} follows from \cite[5.20]{vakil-wood}.
It is natural to consider the $\wbar_{1^k\lambda}$ together for varying $k$, because
they all count polynomials with multiple roots ``at least as bad'' as $\lambda$, as in the examples with $\lambda=d$
and $\lambda=22$ given above.
For any $\lambda$, it is the case that the limit
\begin{equation}\label{E:limit}
\lim_{k\ra\infty} \frac{\wbar_{1^k\lambda}}{q^k}
\end{equation}
exists \cite[Theorem 1.33]{vakil-wood}, but for the partitions $\lambda={b_0}^{e_0}\cdots{b_m}^{e_m}$ satisfying
the hypothesis of Theorem~\ref{T:main}, we see that
in fact $\frac{\wbar_{1^k\lambda}}{q^k}$ is independent of $k$.

To prove Theorem~\ref{T:main}, we start with the idea of the proof of  \cite[Lemma 5.18]{vakil-wood} and add new ideas
that allow us to extend well beyond the cases of $\lambda=1^kab^r$ that \cite{vakil-wood} could prove.
In \cite[5.20]{vakil-wood}, a stronger version of the theorem is proven, one about classes of configuration spaces
in the Grothendieck
ring of varieties over any field (see Section~\ref{S:groth}), and our proof works in that generality as well, resulting
in Theorem~\ref{T:groth}.  

The limits~\eqref{E:limit} have analogous limits for classes of configuration spaces in the Grothendieck ring of varieties, which have
very interesting connections to the homological stabilization of configuration spaces in topology (see \cite[1.41-1.50]{vakil-wood} for more details).  For example, if $\Conf^n X$ is the space of unordered $n$-tuples of distinct points on a 
manifold $X$, then the dimension of the $i$th rational homology group $h_i(\Conf^n X)$ stabilizes for $n$ sufficiently large (given $i$), a recent result of Church \cite{church} and Randal-Williams \cite{randal-williams} for closed manifolds and an older
result of McDuff \cite{McDuff75} for open manifolds.  
In the case that $X=\R^2$, this homological stability is an even older result of Arnol'd \cite{Arnold69}, and moreover,
Arnol'd shows that the $h_i(\Conf^n X)$ are independent of $n$ for $n\geq 2$.  Arnol'd's result is analogous to (could be predicted by) the fact that $w_{1^n}$, or equivalently $\wbar_{1^{n-2}2}$, is independent of $n$ for $n\geq 2$,
and further, the exact values of $h_i(\Conf^n \R^2)$ that Arnol'd gives could be predicted from the exact values of
$\wbar_{1^{n-2}2}$.

Let $\Conf^\lambda_c (X)$ denote the space of unordered tuples of points of a manifold $X$
whose multiplicity partition $\lambda'$ satisfies $\lambda'\not\geq \lambda$.  Informally, $\Conf^\lambda_c (X)$
is the complement of points with multiplicities that are $\lambda$ ``or worse'' (where the ``worse'' configurations are those in
the closure of the configurations with multiplicity $\lambda$).
For example, $\Conf^{1^{n-2}2}_c (X)=\Conf^n X,$ and $\Conf^{1^{n-d}d}_c (X)$ is the space
of unordered sets of $n$ points of $X$ in which all points appear with multiplicity at most $d-1$.
Theorem~\ref{T:main} (and Equation~\eqref{E:simple}) then motivates the following topological conjecture, extending \cite[1.43 Conjecture E]{vakil-wood}.
\begin{conjecture}\label{C:top}
 For integers $m,k\geq 0$, and $b_i,e_i,d\geq 1$ for $0\leq i\leq m$, such that for all $i$,
$b_i\geq\sum_{j<i} e_jb_j$,
we have 
$$
h_\ell(\Conf^{1^k{b_0}^{e_0}\cdots{b_m}^{e_m}}_c(\R^{2d}))=
\begin{cases}
1 &\textrm{$\ell=0$ or $\ell=2d(\sum_i e_i(b_i-1))-1$}\\
0 &\textrm{otherwise.}\
\end{cases}
$$
\end{conjecture}
In the case when $d=1$ and ${b_0}^{e_0}\cdots{b_m}^{e_m}=2$, this is Arnol'd's theorem \cite{Arnold69},
O. Randall-Williams has shown the conjecture for any $d$ when ${b_0}^{e_0}\cdots{b_m}^{e_m}=b_0$, 
and according to T. Church, Arnol'd's work \cite{Arnold70} can be used to show the conjecture for arbitrary $d$ when 
${b_0}^{e_0}\cdots{b_m}^{e_m}=b_0^{e_0}$ (see \cite[Section 1.44]{vakil-wood}). 
Our conjecture goes well beyond the cases that are currently known, and  the recently proven cases were motivated
by \cite[1.43 Conjecture E]{vakil-wood}, a special case of our conjecture, made for the same reasons.

\subsection{Further directions}
It would be interesting to have a complete classification for which $\lambda$ we have that $\frac{\wbar_{1^k\lambda}}{q^k}$ is independent of $k$
(perhaps for $k$ sufficiently large).   Further, we are curious whether the classification is the same as when the dimensions of the homology groups of the analogous configuration spaces are independent of $k$. 
We are also particularly curious as to whether there are examples in which
$\frac{\wbar_{1^k\lambda}}{q^k}$ is independent of $k\gg 0$ but not a power of $q$.
The question of counting polynomials is the case of counting points on the affine line (which gives $X=\A^1_\C=\R^2$ in the topological analog), and we are curious for what other spaces and partitions $\lambda$ does counting points with multiplicity $\lambda$ or worse give this independence in $k$.


\subsection{Outline of the paper}
In Section~\ref{S:notation} we specify our notation for the paper.  In Section~\ref{S:main}, we prove Theorem~\ref{T:main}.
Finally, in Section~\ref{S:groth}, we give the refinement of  Theorem~\ref{T:main} that we have proven on configuration spaces
of any variety in the Grothendieck ring of varieties over any field.

\subsection*{Acknowledgements} 
The first author was supported by National Science Foundation grant DMS-1147782 and the second author was supported by American Institute of Mathematics Five-Year Fellowship and National Science Foundation grant DMS-1147782.

\section{Notation}\label{S:notation}
In this paper, a \emph{partition} $\lambda$ is a multiset, and we use a multiplicative notation so that
$\lambda=a_1^{e_1}\cdots a_m^{e_m}$ is the multiset in which $a_i$ occurs $e_i$ times.  (We avoid two-digit numbers
so that, for example, $\lambda=12$ is the two element multiset including the elements $1$ and $2$.)
We let $|\lambda|=\sum_i e_i$ is the size of the multiset.

Suppose $\lambda,\lambda',$ and $\pi$ are partitions.
If $x,y,z$ are elements with $x+y=z$ such that $\lambda=xy\pi$ and $\lambda'=z\pi$, we
say $\lambda'$ is an {\em elementary merge} of $\lambda$.
In this case $|\lambda|=1+|\lambda'|$. 
We define the {\em refinement ordering} $<$ on 
partitions as generated by elementary merges.  (If $\lambda'$ is
an elementary merge of $\lambda$, then 
$\lambda < \lambda'$.)
 For example, $123 < 3^2 < 6$.  We write  $\la \leq \la'$ if $\la < \la'$
or $\la  = \la'$.

If $\lambda=a_1^{e_1}\cdots a_m^{e_m}$ with the $a_i$ distinct, we could (equivalently to the above) define
$w_\lambda$ to be the number of $m$-tuples $(f_1,\dots,f_m)$ in which $f_i$ is a square-free monic polynomial in $\F_q[x]$ of degree $e_i$ and the $f_i$ are pairwise relatively prime.  
(We can associate to $(f_1,\dots,f_m)$ the polynomial $\prod_i f_i^{a_i}$ with partition $\lambda$).
Again equivalently, we could define
$w_\lambda$ to be the number of assignments  to each monic irreducible $f\in \F_q[x]$
an integer $n_f$ between $0$ and $m$, inclusive, so that $\sum_{f \textrm{ with } n_f=i} \deg(f)=e_i$ for all $i\geq 1$. 
(We can associate such an assignment to a tuple $(f_1,\dots,f_m)$ with $f_i=\prod_{f \textrm{ with } n_f=i} f$.)
In this way we can define $w_\lambda$ for the $a_i$ in any set, not just for $a_i$ positive integers.
Further, we note that $w_\lambda$ only depends on the \emph{multiplicity sequence $e_i$ of $\lambda$}.

\section{Proof of Theorem~\ref{T:main}}\label{S:main}
First we need two lemmas.

\begin{lemma}\label{L:1kdformal}
If $A$ is a formal variable, we have $\wbar_{A^k(bA)}=q^{k+1}$ for all $k\geq 0$ and $b\geq 1$.
\end{lemma}
\begin{proof}
We have  
$$
\wbar_{A^k(bA)}=\sum_{\lambda\geq A^k(bA)} w_{\lambda},
$$
and by dividing each element of each $\lambda$ in the sum by $A$, we see that
$\wbar_{A^k(bA)}=\wbar_{1^kb}$, which is $q^{k+1}$ \cite[Proposition 5.9(b)]{vakil-wood}.  We give a proof here for
completeness.
Let $c_n$ be the number of monic polynomials in $\F_q[x]$ of degree $n$ in which every root appears with multiplicity
at most $b-1$, and $d_n=q^n$ be the number of monic polynomials in $\F_q[x]$ of degree $n$.  Since we can factor any monic $f\in \F_q[x]$ uniquely as $g(x)h(x)^b$, so that
every root of $g$ has multiplicity at most $b-1$ and $g(x)$ and $h(x)$ are both monic, we have an equality of generating functions
$$
(1-tq)^{-1}=\sum_{n} d_nt^n=\sum_{n,m} c_md_n t^{m+bn}.  
$$
Thus
$
\sum_{m} c_m t^m = (1-t^bq)/(1-tq),
$
and so $c_n=q^{n}-q^{n-b+1}$ for $n\geq b$.
The lemma follows because  $\wbar_{1^kb}=\wbar_{1^{k+b}}-c_{k+b}=q^{k+1}$.
\end{proof}

\begin{lemma}[Formal product rule]\label{L:pr}
Let $A,B_0,\dots,B_m$ be formal variables and $b_0,e_0,\dots e_m$ be integers at least $1$.  Then
$$
\wbar_{A^k (b_0A) B_0^{e_0-1} B_1^{e_1}\dots B_m^{e_m}}=
\wbar_{A^k (b_0A)} \wbar_{B_0^{e_0-1}} \wbar_{ B_1^{e_1}}\dots  \wbar_{B_m^{e_m}}.
$$
\end{lemma}
\begin{proof}
We see that each side counts the following: the number of ways to assign to each irreducible monic polynomial $f\in \F_q[x]$
a tuple $(a_fA,b_{f,0}B_0,\dots,b_{f,m}B_m)$ such that 1)$a_f,b_{f,i}$ are non-negative integers,
2)$\sum_f a_f\deg(f)=k+b_0$ and $\sum_f b_{f,0}\deg(f)=e_0-1$ and $\sum_f b_{f,i}\deg(f)=e_i$ for $i>0$, and
3)at least one $a_f$ is at least $b_0$.
On the left-hand side of the lemma, such as assignment corresponds to one element counted by
$w_\lambda$, where $\lambda$ contains the element $nA+n_0B_0+\dots+n_mB_m$ exactly $e$ times, where
$e=\sum_{f \textrm{ with } a_f=n,b_{f,i}=n_i \textrm{ for all $i$}} \deg(f)$.  The element counted is the tuple composed of the square-free polynomials $\prod_{f \textrm{ with } a_f=n,b_{f,i}=n_i \textrm{ for all $i$}} f$.
On the right-hand side of the lemma, such as assignment corresponds to an $m+2$ tuple of elements counted by
$w_{\lambda},w_{\lambda_0},\dots w_{\lambda_m}$, respectively, where, 
$\lambda$ contains the element $nA$ exactly $e$ times, where
$e=\sum_{f \textrm{ with } a_f=n} \deg(f),$ and for all $i$, we have that $\lambda_i$ contains the element $n_iB_i$ exactly $e_i$ times, where
$e_i=\sum_{f \textrm{ with } b_{f,i}=n_i } \deg(f)$. 
\end{proof}

 We fix integers $m\geq -1$, and $k\geq 0$, and $b_i,e_i\geq 1$ for $0\leq i\leq m$, such that for all $i$,
\begin{equation}\label{E:cond}
b_i\geq\sum_{j<i} e_jb_j.
\end{equation}
We will now prove Theorem~\ref{T:main} by induction on $\sum_i e_i$, where the base case $m=-1$ and $\sum_i e_i=0$ is clear.
If $\mu$ is a partition, let $\mathcal R_\mu$ be the set of partitions
$\geq \mu$.
The map $\phi$ that sends $A\mapsto 1$ and $B_i \mapsto b_i$ for formal variables $A,B_i$ induces a map
of posets
$$
\mathcal R_{A^{k+b_0}B_0^{e_0-1}B_1^{e_1}\cdots B_m^{e_m}}\ra \mathcal R_{1^{k+b_0}b_0^{e_0-1}b_1^{e_1}\cdots b_m^{e_m}}.
$$
We will see that $\phi$ restricts to a
{\em bijection} 
\begin{equation}\label{E:map}
\mathcal R_{A^{k+b_0}B_0^{e_0-1}B_1^{e_1}\cdots B_m^{e_m}}\setminus 
\mathcal R_{A^{k}(b_0A)B_0^{e_0-1}B_1^{e_1}\cdots B_m^{e_m}}
\ra \mathcal R_{1^{k+b_0}b_0^{e_0-1}b_1^{e_1}\cdots b_m^{e_m}}\setminus
\mathcal R_{1^{k}b_0^{e_0}b_1^{e_1}\cdots b_m^{e_m}},
\end{equation} 
and that this bijection preserves the multiplicity sequence of each partition.  
We let $\pi_i$ be the map on integers that is reduction to standard representatives modulo $b_i$, and
note that it induces a map on partitions of integers.

\begin{lemma}\label{L:key}
If $0\leq r\leq b_0-1$, and $0\leq s_0\leq e_0-1,$ and $0\leq s_i\leq e_i$ for $i\geq 1$, then
  if we successively apply
$\pi_m,\dots, \pi_l$ to $\phi(rA+\sum_i s_iB_i)$, the
 the map $\pi_l$ reduces $\pi_{l+1}\circ\cdots\circ \pi_m \circ \phi(rA+\sum_i s_iB_i )=
 r+\sum_{i\leq l} s_ib_i
  $
 by exactly $s_lb_l$.
\end{lemma}
\begin{proof}
We induct downwards on $l$.  
We have  $r+\sum_{i\leq l-1} s_ib_i\leq b_0-1  +(\sum_{i\leq l-1} e_ib_i) -b_0\leq b_l-1 $
by Equation~\ref{E:cond}.  Since $r+\sum_{i\leq l-1} s_ib_i\geq 0$, it must be that
$r+\sum_{i\leq l-1} s_ib_i\geq 0$ is the standard reduction of $r+\sum_{i\leq l} s_ib_i$ modulo $b_l$.
\end{proof}

We first see that $\phi$ restricts to a map as in Equation~\ref{E:map}. Suppose for contradiction that
for some $\lambda\in \mathcal R_{A^{k+b_0}B_0^{e_0-1}B_1^{e_1}\cdots B_m^{e_m}}\setminus 
\mathcal R_{A^{k}(b_0A)B_0^{e_0-1}B_1^{e_1}\cdots B_m^{e_m}}$, we have $\phi(\lambda)=\mu\geq 1^{k}b_0^{e_0}b_1^{e_1}\cdots b_m^{e_m}$.
Let the elements of $\mu$ be $\mu_j=r_j+\sum_i s_{j,i} b_i$, with $r_j$ and $s_{j,i}$ non-negative integers,
not all $0$ for a fixed $j$, such that  $\sum_j r_j=k$, and $\sum_j s_{j,i}=e_i$ for all $i$.
We have that $\pi_m(\mu_j)$ reduces $\mu_j$ by at least $s_{j,m} b_m$, and since
we know by Lemma~\ref{L:key} that the total reduction of elements of $\phi(\lambda)=\mu$ is exactly
$e_mb_m$, it must be that $\pi_m(\mu_j)$ reduces $\mu_j$ by exactly $s_{j,m} b_m$.
Similarly, we make the same argument for the successively applied $\pi_{m-1},\dots,\pi_0$, but then we have a contradiction as
$\pi_0$ reduces the elements of $\pi_{1}\circ\cdots\circ \pi_m (\mu)$ by at least $e_0 b_0$ total,
but the elements of $\pi_{1}\circ\cdots\circ \circ \phi(\lambda)$ by $(e_0-1) b_0$ by Lemma~\ref{L:key}.

Next we see that $\phi$ gives a bijection in Equation~\ref{E:map}.  In fact, Lemma~\ref{L:key} has the following corollary.
\begin{corollary}\label{C:inv}
Let $0\leq r,r'\leq b_0-1$, and $0\leq s_0,s_0'\leq e_0-1,$ and $0\leq s_i,s_i'\leq e_i$ for $i\geq 1$. If
$$
 r+\sum_{i} s_ib_i = r'+\sum_{i} s_i'b_i,
$$
then $r=r'$ and $s_i=s_i'$ for all $i$.
\end{corollary}
\begin{proof}
We successively apply $\pi_m,\dots,\pi_0$ to obtain $s_i=s_i'$ by Lemma~\ref{L:key}, and the final remainder is $r=r'$.
\end{proof}
So, we see that if $\lambda,\lambda'\in \mathcal R_{A^{k+b_0}B_0^{e_0-1}B_1^{e_1}\cdots B_m^{e_m}}\setminus 
\mathcal R_{A^{k}(b_0A)B_0^{e_0-1}B_1^{e_1}\cdots B_m^{e_m}},$ with $\phi(\lambda)=\phi(\lambda')$ then $\lambda=\lambda'$,
for if $e,e'$ are elements of $\lambda,\lambda'$ respectively, then $\phi(e)=\phi(e')$ implies $e=e'$ by Corollary~\ref{C:inv}.

Finally, Corollary~\ref{C:inv} implies that $\phi$  in Equation~\ref{E:map} preserves multiplicity sequences of partitions, as for
$\lambda\in \mathcal R_{A^{k+b_0}B_0^{e_0-1}B_1^{e_1}\cdots B_m^{e_m}}\setminus 
\mathcal R_{A^{k}(b_0A)B_0^{e_0-1}B_1^{e_1}\cdots B_m^{e_m}}$ the application of $\phi$ does not make any two unequal elements of $\lambda$ equal.

Since $w_\pi$ only depends on the multiplicity sequence of $\pi$, we have
\begin{align*}
\wbar_{1^{k+b_0}b_0^{e_0-1}b_1^{e_1}\cdots b_m^{e_m}}-\wbar_{1^{k}b_0^{e_0}b_1^{e_1}\cdots b_m^{e_m}}&=
\sum_{\lambda \in \mathcal R_{1^{k+b_0}b_0^{e_0-1}b_1^{e_1}\cdots b_m^{e_m}}\setminus
\mathcal R_{1^{k}b_0^{e_0}b_1^{e_1}\cdots b_m^{e_m}}} w_\lambda\\
&= \sum_{\mu \in \mathcal R_{A^{k+b_0}B_0^{e_0-1}B_1^{e_1}\cdots B_m^{e_m}}\setminus 
\mathcal R_{A^{k}(b_0A)B_0^{e_0-1}B_1^{e_1}\cdots B_m^{e_m}}} w_\mu\\
& =   \wbar_{A^{k+b_0}B_0^{e_0-1}B_1^{e_1}\cdots B_m^{e_m}} -  \wbar_{A^{k}(b_0A)B_0^{e_0-1}B_1^{e_1}\cdots B_m^{e_m}}.
\end{align*}

By Lemmas~\ref{L:1kdformal} and \ref{L:pr}, we have that $\wbar_{A^{k+b_0}B_0^{e_0-1}B_1^{e_1}\cdots B_m^{e_m}}=q^{k+b_0-1+\sum_i e_i}$ and $\wbar_{A^{k}(b_0A)B_0^{e_0-1}B_1^{e_1}\cdots B_m^{e_m}}=
q^{k+\sum_i e_i}.$
By induction, we have that $\wbar_{1^{k+b_0}b_0^{e_0-1}b_1^{e_1}\cdots b_m^{e_m}}=q^{k+b_0-1+\sum_i e_i},$
and so Theorem~\ref{T:main} follows.

\section{In the Grothendieck ring of varieties}\label{S:groth}

\label{s:Gring}

Let  \highlight{$\k$} be a field.  
The Grothendieck ring of varieties $\cM := K_0(\var_{\k})$
is defined as follows.  As an abelian group, it is generated by the
classes of finite type $\k$-schemes up to isomorphism.  The class of
a scheme $X$ in $\cM$ is denoted \highlight{$[X]$}.  
The group relations are generated by
the following ``cut and paste'' relations: if $Y$ is a closed subscheme of $X$, and $U$ is its
(open) complement, then $[X] = [U] + [Y]$. 
 The product $[X][Y] :=[X\times_{\k} Y]$ makes $\cM$
into a commutative ring.  

For a partition $\lambda=a_1^{e_1}\cdots a_m^{e_m}$ with $a_i$ distinct, we  
define  $w_\l( X)$ 
to be  the open  subscheme of
$\prod_i \Sym^{e_i}X$ in which all the points are distinct, i.e.\ the
complement of the ``big diagonal''. 
(Note that taking $\k=\F_q$ applying the $\F_q$-point counting functor to
$w_{\l}(\A^1)$ recovers the integer $w_\l$ defined above.)
Define  \highlight{$\overline{w}_\lambda(X) = \sum_{\l'\geq \lambda }
  [w_{\l'}(X)]$}.

Let \highlight{$\zzeta_X(t) :=\sum_{n\geq 0} [\Sym^n X]t^n
  \in \cM[[t]]$} be the {\em motivic zeta function} (defined by
Kapranov, \cite[(1.3)]{Kapranov2000}).  If $\k =\F_q$, then
the $\F_q$-point counting functor sends $\zzeta_X(t)$ to the Weil zeta function $\zeta_X(s)$, where
$t=q^{-s}$. For a partition $\l$, we define
$
\Kbar_{X,1^\bullet \l}(t):=\sum_{j} \wbar_{1^j \l} (X) t^j\in \cM[[t]].
$
See \cite[1.1-1.11 and Section 2]{vakil-wood} for a more detailed introduction
to the above topics in this context.

As the motivic analog of Theorem~\ref{T:main}, for $\l=b_0^{e_0}\cdots b_m^{e_m}$ satisfying the
hypotheses of Theorem~\ref{T:main}, we will determine $\Kbar_{X,1^\bullet \l}(t)$ in terms of $\zzeta_X(t)$.
The following will replace Lemma~\ref{L:1kdformal}.

\begin{lemma}[Proposition 5.9(b) of \cite{vakil-wood}]\label{P:base}
For an integer $a>1$, we have $$\Kbar_{X,1^{\bullet} a}(t)=t^{-a} \zzeta_X(t) (1-1/\zzeta_X(t^a)).$$
\end{lemma}

\begin{theorem}[Refinement of Theorem~\ref{T:main} in the Grothendieck ring]\label{T:groth}
For a variety $X$ over $\k$, and integers $m\geq -1,$ and $k\geq 0$, and $b_i,e_i\geq 1$ for $0\leq i\leq m$, such that each $1\leq i\leq m$, we have
$b_i\geq\sum_{j<i} e_jb_j$, we have
\begin{enumerate}
 \item \label{E:rec}
for formal variables $A,B_i$, and $m\geq 0$, we have
 \begin{align*}
\wbar_{1^{k}b_0^{e_0}b_1^{e_1}\cdots b_m^{e_m}}(X)=
\wbar_{1^{k+b_0}b_0^{e_0-1}b_1^{e_1}\cdots b_m^{e_m}}(X)
 -  \wbar_{A^{k+b_0}B_0^{e_0-1}B_1^{e_1}\cdots B_m^{e_m}}(X) +  \wbar_{A^{k}(b_0A)B_0^{e_0-1}B_1^{e_1}\cdots B_m^{e_m}}(X),
\end{align*}
and
$$
\Kbar_{X,1^{\bullet} b_0^{e_0}b_1^{e_1}\cdots b_m^{e_m}} (t)
 =  \Kbar_{X,1^{\bullet} b_0^{e_0-1}b_1^{e_1}\cdots b_m^{e_m}} (t) t^{-{b_0}} - \frac {\zzeta_X(t)t^{-b_0}} {\zzeta_X(t^{b_0})}
\left[  \Sym^{e_0-1} X \times \prod_{i=1}^m \Sym^{e_i} X \right].$$
\item\label{E:formula} $$
\overline{K}_{X,1^{\bullet} b_0^{e_0}b_1^{e_1}\cdots b_m^{e_m}}(t) = t^{-\sum_i e_ib_i}\left( \zzeta_X(t)-
\sum_{i=0}^{m} \frac{\zzeta_X(t)}{\zzeta_X(t^{b_i})}  \prod_{l=i+1}^m [\Sym^{e_l} X] t^{e_lb_l} \sum_{j=0}^{e_i-1} [\Sym^{j}X ]t^{jb_i}
\right).
$$
\end{enumerate}
\end{theorem}

In the special cases when $m=1$ or both $m=2$ and $e_0=1$, Theorem~\ref{T:groth} reduces to \cite[Lemma 5.18, Proposition 5.19, Example 5.20]{vakil-wood}.  Taking $\k=\F_q$ and applying the $\F_q$-point counting functor to
Theorem~\ref{T:groth} \eqref{E:formula} with $X=\A^1$ gives Theorem~\ref{T:main} (using the basic fact
that $\Sym^a \A^1$ has $q^a$ points for $a\geq 0$).  As $[\Sym^r \A^d]=[\A^{rd}]$ (e.g. see \cite[Lemma 4.4]{Got01}), 
Theorem~\ref{T:groth} \eqref{E:formula} with $X=\A^d$ gives a very similar result to that of $X=\A^1$, just
with each $[\A^s]$ replaced by $[\A^{sd}]$, and we have
\begin{equation}\label{E:simple}
 \wbar_{1^k{b_0}^{e_0}\cdots{b_m}^{e_m}}(\A^d) = [\A^{d(k+\sum_i e_i)}]
\end{equation}
for any $d\geq 0$.
 These are the cases that motivate Conjecture~\ref{C:top}
(see \cite[1.41-1.44]{vakil-wood} for more details about this motivation).

\begin{proof}
 The first part of \eqref{E:rec} follows exactly as the same statement in the proof of Theorem~\ref{T:main}.
The second follows by multiplying both sides of the first by $t^k$, summing over $k$, and applying Lemma~\ref{L:pr} (which has an analogous proof in the Grothendieck ring setting) and Lemma~\ref{P:base}.
Finally, \eqref{E:formula} is proven inductively using Lemma~\ref{P:base} as a base case and the second part of 
\eqref{E:rec} for the inductive step.
\end{proof}

\bibliographystyle{alpha}
\bibliography{myrefs.bib}

\end{document}